\documentclass[11pt]{amsart}
\usepackage{amsmath, amsthm, amssymb, amsfonts, mathtools}
\usepackage[T1]{fontenc}
\newif\ifscalablefonts\scalablefontsfalse
\IfFileExists{libertinus.sty}{%
  \usepackage{libertinus}\scalablefontstrue
  \IfFileExists{newtxmath.sty}{\usepackage[libertine,vvarbb]{newtxmath}}{}%
}{%
  \IfFileExists{libertine.sty}{%
    \usepackage{libertine}\scalablefontstrue
    \IfFileExists{newtxmath.sty}{\usepackage[libertine,vvarbb]{newtxmath}}{}%
  }{\IfFileExists{lmodern.sty}{\usepackage{lmodern}\scalablefontstrue}{}}%
}
\IfFileExists{microtype.sty}{%
  \ifscalablefonts\usepackage[protrusion=true,expansion=true,final]{microtype}%
  \else\usepackage[protrusion=true,expansion=false,final]{microtype}\fi}{}
\usepackage[svgnames]{xcolor}
\usepackage[margin=1.35in]{geometry}
\usepackage{booktabs}
\usepackage[hang,flushmargin]{footmisc}
\usepackage{enumitem}
\setlist{nosep,leftmargin=2em,labelsep=0.6em}
\setlist[itemize]{label=\textendash}
\usepackage{tikz-cd}
\usepackage{tikz}
\usetikzlibrary{arrows.meta,calc,decorations.pathreplacing}
\usepackage{pgfplots}\pgfplotsset{compat=1.17}
\usepackage{caption}
\usepackage{thmtools}
\definecolor{dark-red}{rgb}{0.4,0.15,0.15}
\definecolor{cite-red}{rgb}{0.55,0.22,0.18}
\usepackage[pagebackref]{hyperref}
\usepackage{hyperxmp}
\hypersetup{colorlinks, linkcolor=dark-red, citecolor=cite-red, urlcolor=cite-red,
  pdftitle={Pontryagin-Thom, the connected components of the motivic sphere spectrum, and principal quadratic divisors},
  pdfauthor={Kyle Ormsby},
  pdfkeywords={motivic homotopy theory, framed correspondences, Grothendieck-Witt ring, Burnside ring, Pontryagin-Thom construction, reciprocity law},
  pdfsubject={The connected components of the motivic sphere spectrum via a Segal-style Pontryagin-Thom argument},
  pdfcontactemail={kyle@e-infinity.space},
  pdfcontactcity={Portland}, pdfcontactcountry={USA}}
\usepackage[capitalize,noabbrev]{cleveref}
\crefname{inputt}{Input}{Inputs}
\crefformat{section}{\S#2#1#3}
\Crefformat{section}{Section~#2#1#3}
\crefrangeformat{section}{\S\S#3#1#4--#5#2#6}
\crefmultiformat{section}{\S\S#2#1#3}{ and #2#1#3}{, #2#1#3}{ and #2#1#3}
\crefformat{subsection}{\S#2#1#3}
\Crefformat{subsection}{\S#2#1#3}
\crefrangeformat{subsection}{\S\S#3#1#4--#5#2#6}
\crefmultiformat{subsection}{\S\S#2#1#3}{ and #2#1#3}{, #2#1#3}{ and #2#1#3}
\Crefname{inputt}{Input}{Inputs}
\crefname{equation}{}{}
\creflabelformat{equation}{\textup{(#2#1#3)}}
\renewcommand*{\backref}[1]{}
\renewcommand*{\backrefalt}[4]{%
  \ifcase #1\relax\or{\footnotesize\textcolor{cite-red!80!black}{(Cited on p.~#2.)}}%
  \else{\footnotesize\textcolor{cite-red!80!black}{(Cited on pp.~#2.)}}\fi}

\newcommand{\Z}{\mathbb{Z}}
\newcommand{\Q}{\mathbb{Q}}
\newcommand{\R}{\mathbb{R}}
\newcommand{\C}{\mathbb{C}}
\newcommand{\A}{\mathbb{A}}
\newcommand{\PP}{\mathbb{P}}
\newcommand{\sph}{\mathbb{S}}
\newcommand{\GW}{\mathrm{GW}}
\newcommand{\W}{\mathrm{W}}
\newcommand{\form}[1]{\langle #1\rangle}
\newcommand{\rk}{\operatorname{rank}}
\newcommand{\tr}{\operatorname{tr}}
\newcommand{\Fr}{\mathrm{Fr}}
\newcommand{\ZF}{\mathrm{ZF}}
\newcommand{\pfr}{\pi_0^{\mathrm{fr}}}
\newcommand{\Hilbfr}{\mathrm{Hilb}^{\mathrm{fr}}}
\newcommand{\Au}{A^{u}}

\DeclareMathOperator{\Hom}{Hom}
\newcommand{\PT}{\mathrm{PT}}
\newcommand{\Bez}{\mathrm{B\acute{e}z}}
\newcommand{\res}{\operatorname{res}}
\newcommand{\disc}{\operatorname{disc}}
\DeclareMathOperator{\coker}{coker}

\declaretheoremstyle[
  headfont=\bfseries, notefont=\mdseries, notebraces={(}{)},
  bodyfont=\itshape, headpunct={.}, postheadspace=0.6em]{plainstyle}
\declaretheoremstyle[
  headfont=\bfseries, notefont=\mdseries, notebraces={(}{)},
  bodyfont=\normalfont, headpunct={.}, postheadspace=0.6em]{defstyle}
\declaretheorem[style=plainstyle,numberwithin=section]{theorem}
\declaretheorem[style=plainstyle,sibling=theorem]{proposition}
\declaretheorem[style=plainstyle,sibling=theorem]{lemma}

\declaretheorem[style=defstyle,sibling=theorem]{definition}
\declaretheorem[style=defstyle,sibling=theorem]{example}
\declaretheorem[style=defstyle,sibling=theorem]{remark}
\declaretheorem[style=plainstyle,sibling=theorem,name=Input]{inputt}

\begin{document}

\title[Pontryagin--Thom, $\pi_0$ of the motivic sphere, and principal quadratic
divisors]{Pontryagin--Thom, \\$\pi_0$ of the motivic
sphere spectrum, \\and principal quadratic divisors}

\author{Kyle Ormsby}
\address{Department of Mathematics \& Statistics, Reed College, Portland, OR 97202, USA}
\email{kyle@e-infinity.space}
\dedicatory{Dedicated to Igor Kriz on the occasion of his 60th birthday.}
\date{\today}

\begin{abstract}
Segal computed $\pi_0$ of the equivariant sphere spectrum as the Burnside
ring by a Pontryagin--Thom construction, transversality, and marks. We carry
out the same argument in the motivic setting over a field $k$: the Burnside ring is replaced by the
Grothendieck ring $\Au(\Gamma_k)$ of finite \'etale $k$-schemes decorated
with a unit, the marks by the Dress map $d(L,u)=\tr_{L/k}\form u$ into
$\GW(k)$, and a Pontryagin--Thom map $\Au(\Gamma_k)\to\pi_0(\sph_k)(k)$ is
shown to be surjective with kernel $\ker(d)$. This gives
$\pi_0(\sph_k)(k)\cong\GW(k)$ for infinite perfect fields of characteristic
$\ne2$ without Morel's computation. We then identify $\ker(d)$, by a
reciprocity law on $\PP^1$, as the group of principal quadratic divisors of
rational functions, so that $\pi_0$ of the motivic sphere is the
unit-decorated Burnside ring of $\Gamma_k$ modulo principal quadratic
divisors. The transversality step is proved by an elementary
``re-topping'' argument.
\end{abstract}

\maketitle

\begin{flushright}\small\itshape
Omnia sponte fluant, absit violentia rebus.\\[1pt]
\upshape --- J.\,A.\ Comenius, \emph{Opera didactica omnia} (1657)
\end{flushright}
\vspace{0.5em}

\setcounter{tocdepth}{1}
{\small\tableofcontents}

\section{Introduction}\label{sec:intro}

\subsection{Segal's theorem}
Let $G$ be a finite group and $A(G)$ its Burnside ring. Segal's theorem
\cite{Segal71,tomDieck} says that the Pontryagin--Thom construction
$A(G)\to\pi_0^G(\sph)$ is a ring isomorphism, and its proof has five steps:
(S1) embed a finite $G$-set $S$ in a representation $V$ and collapse the
complement of a tubular neighborhood, $S^V\to\bigvee_SS^V\to S^V$; (S2) any
two embeddings are isotopic after stabilizing, so the class is well
defined; (S3) disjoint union and product go to sum and product; (S4) by
equivariant transversality every $f\colon S^V\to S^V$ is homotopic to the
collapse of the finite $G$-set $f^{-1}(0)$; (S5) the fixed-point degrees
$\pi_0^G(\sph)\to\prod_{(H)}\Z$ compose with $\PT$ to the mark homomorphism
$[S]\mapsto(|S^H|)$, which is injective by Burnside's theorem. The shape is:
\emph{build the map out of a combinatorial ring, prove surjectivity by a
transversality theorem, prove injectivity by a computable invariant already
known to detect the combinatorial ring.} This note carries out that shape
in the stable motivic homotopy category over a field.

\subsection{The source}\label{sec:shape}
What should play the role of $A(G)$ is dictated by the mechanism. The
Pontryagin--Thom map is the comparison from the free semiadditive category
on a point --- spans $X\leftarrow Z\to Y$ with $Z\to X$ finite, which in
the discrete setting is the universality of finite spans \cite{Harpaz} and
in the motivic one the recognition theorem for framed correspondences
\cite{EHKSY} --- to the unit of the target. A wrong-way map along $Z\to X$
exists in the sphere only when $Z\to X$ is \emph{framed}, \emph{i.e.}\ carries a
trivialization of its cotangent complex \cite{EHKSYtransfers}, so the
backward leg must carry that datum, and the decoration of an orbit is
$\pi_0$ of the torsor of its framings. Topologically $\pi_0GL_n(\R)=\Z/2$,
and a point with the non-canonical framing is already a $\Z$-combination
of canonically framed orbits (\cref{sec:tomdieck}), which is why Segal's
source is free on orbits. Motivically $\pi_0^{\A^1}GL_n=\mathbb G_m$ by the
determinant ($SL_n$ being $\A^1$-connected \cite{Morel12}; this heuristic is
the only place Morel's results enter, and the proofs use only that
$SL_n=E_n$ over commutative semilocal rings), and $\form u\ne\pm\form1$, so
orbits must be decorated by units.

\begin{definition}\label{def:Au}
The \emph{unit-decorated Burnside ring} $\Au(\Gamma_k)$ is the Grothendieck
ring of pairs $(X,u)$ with $X$ a finite \'etale $k$-scheme and
$u\in\mathcal O(X)^\times$, under disjoint union and $(X,u)\otimes(Y,v)=
(X\times_kY,\mathrm{pr}^*u\cdot\mathrm{pr}^*v)$. As a group it is free on
the classes $(L,u)$, $L/k$ finite separable and $u\in L^\times$, up to
$k$-isomorphism; it is Dress's monomial Burnside ring
$\Omega(\Gamma_k,(k^{\mathrm{sep}})^\times)$ \cite{Dress71}, and $u=1$
embeds the ordinary Burnside ring $A(\Gamma_k)$. No relation among units is
imposed: $\form u=\form{us^2}$ is part of the answer, not of the input.
The \emph{Dress map} is the ring homomorphism
\[
d\colon\Au(\Gamma_k)\to\GW(k),\qquad d(L,u)=\tr_{L/k}\form u,
\]
the transfer of the rank-one form $\form u$ along the trace; restricted to
$u=1$ it is the map through which $\GW$ is a Green functor over the
Burnside ring \cite{Dress75}, whence the name.
\end{definition}

\subsection{Results}\label{sec:mainstatement}

\begin{restatable}{theorem}{mainthm}\label{thm:main}
Let $k$ be an infinite perfect field of characteristic $\ne2$. The
Pontryagin--Thom map
\[
\PT\colon\Au(\Gamma_k)\longrightarrow\pi_0(\sph_k)(k),
\]
sending $(L,u)$ to the class of a framed point with support
$\operatorname{Spec}L$ and framing determinant $u$ (\cref{sec:pt}), is a
surjective ring homomorphism with kernel $\ker(d)$. Hence
$\pi_0(\sph_k)(k)\cong\Au(\Gamma_k)/\ker(d)\cong\GW(k)$, the last
isomorphism being the Scheja--Storch degree $w$ of \cref{sec:w}.
\end{restatable}

The proof uses the framed recognition theorem of Garkusha--Panin and
Elmanto--Hoyois--Khan--Sosnilo--Yakerson, Scheja--Storch duality, Harder's
theorem on forms over $k[t]$, Cazanave's classification of pointed rational
functions, and Witt's presentation of $\GW(k)$; it does not use Morel's
computation of $\pi_0$, nor the Milnor--Witt machinery through which the
Milnor conjecture enters the graded statement.

Our second result makes
the presentation explicit. For $f\in k(t)^\times$ the \emph{quadratic
divisor} $\rho(f)\in\Au(\Gamma_k)$ records the closed points of $\PP^1$ at
which $f$ has odd valuation, each decorated by the leading coefficient of
$f$ there, twisted by the Euler unit $\pi_p'(\theta_p)$; the
\emph{principal quadratic divisor} is $R(f)=\rho(f)-r(f)h_{\!A}$, corrected
by hyperbolic classes $h_{\!A}=(k,1)+(k,-1)$ to have rank zero
(\cref{def:R}).

\begin{theorem}\label{thm:reciprocity}
For $k$ perfect of characteristic $\ne2$, $\ker(d)$ is the subgroup $R$
generated by the principal quadratic divisors of rational functions on
$\PP^1_k$. Consequently, for $k$ moreover infinite, $\PT$ induces a ring
isomorphism
\[
\Au(\Gamma_k)\big/\big(\text{principal quadratic divisors}\big)
\ \xrightarrow{\ \sim\ }\ \pi_0(\sph_k)(k).
\]
\end{theorem}

The relations come from a reciprocity law on $\PP^1$ at the level of
$\GW$ (\cref{prop:recipP1}), which we prove directly and which
is, read through the degree of \cref{sec:w}, the statement that the
Scheja--Storch form of a framed point of the line is determined by its
behavior at infinity.

\subsection{The dictionary}
The two arguments correspond line by line; the table is the organizing
principle of the note and the reader may use it as a map.

\begin{center}\footnotesize
\begin{tabular}{p{0.29\textwidth}p{0.47\textwidth}p{0.15\textwidth}}
\toprule
Segal & motivic & here\\
\midrule
finite group $G$ & $\Gamma_k$ & \\
orbit $G/H$ & $\operatorname{Spec}L$, $L/k$ finite separable & \\
framing of a point & unit $u\in L^\times$, via the framing determinant &
\cref{sec:pt}\\
Burnside ring $A(G)$ & $\Au(\Gamma_k)$ & Def.\ \ref{def:Au}\\
$\pi_0^G(\sph)=$ framed $G$-bordism & framed recognition &
\cref{inp:rec}\\
(S1) collapse construction & framed correspondence & \cref{sec:pt}\\
(S2) isotopy of embeddings & shears, framing rigidity &
Thm.\ \ref{thm:rigid}\\
(S3) ring map & external product & Prop.\ \ref{prop:ptring}\\
fixed-point degree & Scheja--Storch form $w$ & Thm.\ \ref{thm:w}\\
marks $|S^H|$ & Dress map $d(L,u)=\tr_{L/k}\form u$ & Prop.\ \ref{prop:wpt}\\
(S4) equivariant transversality & Position Lemma & Thm.\ \ref{thm:position}\\
\quad --- & re-topping; density of the \'etale locus &
\cref{sec:retop,sec:moving}\\
(S5) Burnside: marks injective & every \'etale class is split &
Prop.\ \ref{prop:etalesplit}\\
\addlinespace
$A(G)\cong\pi_0^G(\sph)$ & $\GW(k)\cong\pi_0(\sph_k)(k)$ &
Thm.\ \ref{thm:main}\\
\bottomrule
\end{tabular}
\end{center}

Two rows differ in substance. The line marked ``---'' has no counterpart
in Segal: a transverse preimage in topology is a manifold, whereas a framed
zero-dimensional scheme may be non-reduced, and an extra theorem is needed
to show that this fatness carries no information in $\pi_0$. And Burnside's
theorem says the mark homomorphism is injective, whereas the Dress map has
a large kernel; so the isomorphism sits one step to the right:
\[
\begin{tikzcd}[column sep=2.6em, row sep=2.4em]
\text{Segal:} & A(G) \arrow[r, "\PT", "\cong"'] & \pi_0^G(\sph)
\arrow[r, "\text{degrees}", "\text{injective}"'] & \prod_{(H)}\Z \\
\text{here:} & \Au(\Gamma_k) \arrow[r, "\PT", "\text{onto}"']
\arrow[rr, bend right=22, "d"'] & \pi_0(\sph_k)(k)
\arrow[r, "w", "\cong"'] & \GW(k).
\arrow[l, "s"', bend right=35, dashed]
\end{tikzcd}
\]
What makes the second row work is the section $s$ of $w$ with
$\PT=s\circ d$ (\cref{thm:marks}), which has no counterpart in the first. \cref{thm:reciprocity} is
the description of that kernel; \cref{sec:complements} explains why a
kernel is to be expected and checks the picture over $\R$, where it reduces
to Segal's.

\subsection{Relation to other work}
The original computation of $\pi_0(\sph_k)(k)$ is Morel's
\cite{Morel04,Morel12}. The identification of the Scheja--Storch form with
the local $\A^1$-degree and of the B\'ezout form with residue pairings are
theorems of Kass--Wickelgren \cite{KassWickelgren19,KassWickelgren20}, with
the trace formula of \cref{prop:wpt} also in \cite{BBMMO,BMP};
we use only the elementary halves of these statements. Neshitov
\cite{Neshitov18} computed $\pfr(k)$ by framed correspondences in
characteristic zero by realizing the Milnor--Witt relations, and
Druzhinin--Kylling \cite[Thm.\ 5.1]{DK} extended this to perfect fields
of characteristic $\ne2$; both pass through the Position Lemma in a
stronger form, with reduced rational support. The present proof outsources the
relation calculus to Cazanave \cite{Cazanave} and the well-definedness of
the degree to Scheja--Storch and Harder, in the spirit of Barge--Lannes
\cite{BargeLannes}. What is new is the identification of the kernel as the
principal quadratic divisors, the re-topping proof of transversality, and
the organization around $\Au(\Gamma_k)$, which turns the analogy with Segal
into a correspondence.

\subsection{Conventions}\label{sec:conventions}
Throughout, $k$ is an infinite perfect field with
$\operatorname{char}k\ne2$,\footnote{Infiniteness of $k$ is used in \cref{inp:rec} as cited
(\cite{EHKSY} remove it in the tangentially framed model), in the split
configurations of \cref{sec:marks}, and in the genericity arguments of
\cref{sec:retop}, which fail over finite fields --- Poonen's Bertini
theorem prescribes Taylor coefficients but not leading forms.
Characteristic $\ne2$ is used throughout the quadratic form theory
(Harder, diagonalization, Witt's presentation, the binary facts in
\cref{prop:relations}, Cazanave). Perfectness is used in \cref{inp:rec}
and, for a different reason, in \cref{sec:reciprocity}; it is not used in
the Position Lemma, and the \'etale supports that lemma produces have
separable residue fields by definition.} except that \cref{thm:reciprocity} needs only
perfectness and characteristic $\ne2$. We write $\GW(k)$ for the Grothendieck--Witt ring, $\form u$ for the form
$(x,y)\mapsto uxy$, $h=\form1+\form{-1}$, $\W(k)=\GW(k)/\Z h$; $\det$ of a
form is the determinant of a Gram matrix in $k^\times/(k^\times)^2$; $\rk$
is rank. The Iverson bracket $[P]$ is $1$ if $P$ holds and $0$ otherwise.
Proofs that are routine are sketched or omitted; the full details are in
the long version of this paper.

\subsection*{Acknowledgments}
The author warmly thanks Igor Kriz for the mentorship and mathematics he has shared over the years. Kriz's intellectual bravery and zest for pioneering work have been a constant inspiration.

This manuscript is based upon work supported by the National Science Foundation under Grant No.~2604878. Any opinions, findings, and conclusions or recommendations expressed in this material are those of the author and do not necessarily reflect the views of the National Science Foundation.

\subsection*{AI use disclosure}
The program of this paper --- Segal's argument carried out over a field,
with $\Au(\Gamma_k)$ as the source and the Dress map in the role of the
marks --- is the author's. The mathematics was then developed, between
July and October 2026, in extended dialogue with AI assistants
(Anthropic's Claude Fable~5.1 and Claude Opus~5.5). In particular, the
re-topping proof of the Position Lemma (\cref{sec:retop}) was proposed by
the assistant and verified and revised by the author, and the
identification of $\ker d$ with the principal quadratic divisors
(\cref{sec:reciprocity}), including the reciprocity law on $\PP^1$ and
its normalizations, was worked out jointly, with the normalizations
checked by computer. The assistants also carried out several adversarial review passes
on intermediate drafts; those passes found errors --- among them a false
worked example and a gap in the comparison of embeddings --- which were
then corrected. The author has checked every statement and proof and
takes full responsibility for the mathematical content and for the
bibliography.

\section{Framed correspondences and the Pontryagin--Thom map}\label{sec:pt}

\subsection{Framed correspondences}
An \'etale neighborhood of a closed subscheme $Z\subset\A^n_k$ is an
\'etale $U\to\A^n_k$ with a closed immersion $Z\hookrightarrow U$ over
$\A^n$; shrinking $U$ one may assume $U\times_{\A^n}Z=Z$. A \emph{framed
correspondence of level $n$} from $\operatorname{Spec}k$ to itself is a
triple $c=(Z,U,\varphi)$ with $Z\subset\A^n_k$ a closed subscheme finite
over $k$, $U$ an \'etale neighborhood of $Z$, and
$\varphi\colon U\to\A^n$ with $Z=V(\varphi)$ as subschemes of $U$; two
triples that agree on a common refinement of the neighborhood are
identified. Write $\Fr_n(k)$ for the set of these, $\ZF_n(k)$ for the free
abelian group on it modulo $(Z'\sqcup Z'',U,\varphi)=(Z',U\smallsetminus
Z'',\varphi)+(Z'',U\smallsetminus Z',\varphi)$, $\sigma(Z,U,\varphi)=
(Z\times0,U\times\A^1,(\varphi,t))$ for stabilization, and
$\ZF(k)=\operatorname{colim}_\sigma\ZF_n(k)$. The same definitions over the
base $\A^1_k$, with supports \emph{finite over $\A^1$} --- the crucial
condition throughout --- give $\ZF(\A^1_k)$, and
\[
\pfr(k):=\coker\big(\ZF(\A^1_k)\xrightarrow{\mathrm{ev}_0-\mathrm{ev}_1}
\ZF(k)\big).
\]
Elements of $\ZF(\A^1_k)$ are \emph{$\A^1$-homotopies}. The external
product $(Z\times Z',U\times U',(\varphi,\varphi'))$ makes $\pfr(k)$ a
commutative ring with unit $[(\{0\},\A^1,x)]$. Commutativity holds only
up to a permutation of coordinates, and that is absorbed because
$GL_n(k)$, acting by linear changes of coordinates on $\A^n$ and on the
framing, acts trivially on $\pfr(k)$: an elementary matrix acts trivially
by the straight-line homotopy $e_{ij}(t\lambda)$, and
$\mathrm{diag}(1,\dots,1,\lambda)$ acts trivially because it rescales the
stabilizing coordinate, which an isomorphism of correspondences undoes.
Finally, write $T:=\A^1/(\A^1\smallsetminus0)$ and
$T^n:=\A^n/(\A^n\smallsetminus0)\simeq T^{\wedge n}$ for the motivic
spheres, and $\pi_0(\sph_k)(k):=\Hom_{\mathcal{SH}(k)}(\sph_k,\sph_k)$.

\begin{inputt}[Framed recognition \cite{GP21,EHKSY}]\label{inp:rec}
For $k$ infinite and perfect, sending $(Z,U,\varphi)$ to the composite
$T^n\to\A^n/(\A^n\smallsetminus Z)\simeq U/(U\smallsetminus Z)
\xrightarrow{\varphi}T^n$ induces a ring isomorphism
$\pfr(k)\cong\pi_0(\sph_k)(k)$.
\end{inputt}

This rests on strict homotopy invariance \cite{GP20} and cancellation
\cite{AGP}; the map is Voevodsky's \cite{VoevodskyNotes}, under which the
external product is the smash product of maps, which agrees with
composition on $\pi_0$. It is the analogue of identifying
$\pi_0^G(\sph)$ with $0$-dimensional framed $G$-bordism, and we use it as a
black box: nothing below uses a computation of $\pi_0$. (The literature
also uses closed \emph{subsets} as supports; the two conventions give the
same $\pfr(k)$ \cite[\S2.1]{EHKSY}.)

\subsection{Framing rigidity}
We write $\Omega_U=\Omega_{U/k}$ for the sheaf of K\"ahler differentials of
a $k$-scheme $U$, and recall that for $U\to\A^n$ \'etale the pullbacks
$dx_1,\dots,dx_n$ of the coordinate differentials form a basis of
$\Omega_U$, and that for a closed subscheme $Z\subset U$ with ideal sheaf
$I$ there is the conormal exact sequence
$I/I^2\to\Omega_U|_Z\to\Omega_{Z/k}\to0$, whose first map sends the class
of a function to its differential. For $c=(Z,U,\varphi)$ the Jacobian
$d\varphi=(\partial\varphi_i/\partial x_j)$ is the matrix of the
differentials $d\varphi_i$ in the basis $dx_j$, and its restriction
$\det d\varphi|_Z\in\mathcal O(Z)$ is the \emph{framing determinant} of
$c$. Since the $\varphi_i$ generate $I$, the conormal sequence reads
$\Omega_{Z/k}=\mathcal O_Z^n/(\text{rows of }d\varphi|_Z)$, so the framing
determinant is a unit if and only if $\Omega_{Z/k}=0$, i.e.\ if and only if
$Z$ is \'etale over $k$; we then speak of an \emph{\'etale support}.
Stabilization and linear changes of coordinates preserve the framing
determinant. The algebraic substitute for ``a
framing is determined up to homotopy by its determinant'' is:

\begin{theorem}[Framing rigidity]\label{thm:rigid}
Let $(Z,U,\varphi)$ and $(Z,U',\varphi')$ be framed correspondences with
the same support $Z\hookrightarrow\A^n$. On a common affine \'etale
neighborhood write $\varphi=A_0\varphi'$. Then $A_0$ is invertible near $Z$,
$u:=\det A_0|_Z\in\mathcal O(Z)^\times$ is well defined, and $(Z,U,\varphi)$
is $\A^1$-equivalent to the $u$-twist $(Z,U',(\tilde u\varphi'_1,
\varphi'_2,\dots,\varphi'_n))$ of $(Z,U',\varphi')$, for any lift $\tilde u$
of $u$; this twist depends only on $\bar u\in\mathcal O(Z_{\mathrm{red}})
^\times$. In particular, two framings of an \'etale support with the same
framing determinant give the same class, and every unit of
$\mathcal O(Z)$ is a framing determinant.
\end{theorem}

\begin{proof}[Sketch]
On the common neighborhood $Z$ is cut out by both regular sequences, so the
Koszul complexes are exact; writing also $\varphi'=B_0\varphi$, the rows of
$A_0B_0-1$ are syzygies of $\varphi$, hence have entries in $(\varphi)=I_Z$,
so $A_0$ is invertible along $Z$ and $A_0\bmod I_Z$ is independent of
choices. Localizing at the finitely many points of $Z$ gives a commutative
semilocal ring, over which $SL_n=E_n$ \cite[Ch.\ V, Cor.\ (9.2)]{Bass68}, so
$A_0\,\mathrm{diag}(\det A_0^{-1},1,\dots)$ is a product of elementary
matrices on a neighborhood of $Z$; each elementary factor
$e_{ij}(\lambda)$ is undone by the homotopy $e_{ij}(t\lambda)\varphi'$,
whose support is $Z\times\A^1$ for every $t$. The twist by
$\tilde u(1+i)$, $i\in I_Z$, or by $1+\nu$ with $\nu$ nilpotent on $Z$, is
likewise undone by $\tilde u(1+si)$, resp.\ $(1+s\tilde\nu)\tilde u$, which
are invertible near $Z\times\A^1$ with constant support. For the last
sentence, $d\varphi|_Z=A_0|_Z\,d\varphi'|_Z$.
\end{proof}

Two closed immersions $j_0,j_1\colon\operatorname{Spec}L\hookrightarrow
\A^N$ with framings of equal determinant also become $\A^1$-equivalent after
stabilization: stabilize both $N$ times into $\A^{2N}$, place them as
$(j_0,0)$ with framing $(\varphi_0(x),y)$ and $(0,j_1)$ with framing
$(x,\varphi_1(y))$, and slide each to the common support $(j_0,j_1)$ by the
homotopies $y-tP(x)$, $x-tQ(y)$ with $P(j_0)=j_1$, $Q(j_1)=j_0$ polynomial;
the resulting framings have block-triangular Jacobians with determinants
$\det d\varphi_0$ and $\det d\varphi_1$, and \cref{thm:rigid}
applies. Hence:

\begin{definition}\label{def:eLu}
For $L/k$ finite separable and $u\in L^\times$ let $e(L,u)\in\pfr(k)$ be
the class of any framed correspondence with support $\operatorname{Spec}L$
and framing determinant $u$. Such correspondences exist: for
$L=k(\theta)$ with minimal polynomial $f$ of degree $d$, the element
$u/f'(\theta)$ of $L=k[x]/(f)$ is represented by a unique polynomial $B$
of degree $<d$, which is nonzero and prime to the irreducible $f$, so that
$\big(V(f),\A^1\smallsetminus V(B),Bf\big)$ is a framed correspondence
with framing determinant $(Bf)'(\theta)=B'(\theta)f(\theta)+B(\theta)
f'(\theta)=B(\theta)f'(\theta)=u$, the first term vanishing because
$f(\theta)=0$. Let $E\subseteq\pfr(k)$ be the subgroup generated
by classes with \'etale support; it is generated by the $e(L,u)$.
\end{definition}

\begin{proposition}\label{prop:ptring}
The assignment $(L,u)\mapsto e(L,u)$ extends to a ring homomorphism
$\PT\colon\Au(\Gamma_k)\to\pfr(k)$ whose image is $E$. Composing with
\cref{inp:rec} gives the map of \cref{thm:main}.
\end{proposition}

\begin{proof}
Since $e(L,u)=e(L,\sigma u)$ for $\sigma\in\operatorname{Aut}_k(L)$, the
map $\PT$ is defined on isomorphism classes; it is additive by the defining relation of
$\ZF_n$, and multiplicative because the external product of framed points
with supports $L,L'$ and determinants $u,u'$ has \'etale support
$\operatorname{Spec}(L\otimes_kL')$ and determinant $u\otimes u'$.
\end{proof}

Steps (S1)--(S3) are complete, with $SL_n=E_n$ over commutative semilocal
rings in the role of the connectedness of $GL_n(\R)^+$. Nothing so far uses a hypothesis
on $k$.

\section{The degree}\label{sec:w}

Segal's invariant, the fixed-point degree, is elementary; its motivic
counterpart is the Scheja--Storch form, and its homotopy invariance is
1970s commutative algebra and quadratic form theory.

Let $c=(Z,U,\varphi)\in\Fr_n(k)$ with $U$ affine (possible, as $U$ is
quasi-affine by Zariski's main theorem) and $A:=\mathcal O(Z)=\mathcal
O(U)/(\varphi)$. Write $X_j=x_j\otimes1$ and $Y_j=1\otimes x_j$ for the
two copies of the coordinates on $U\times U$, so that a function
$\varphi_i$ on $U$ gives the two functions $\varphi_i(X)=\varphi_i\otimes1$
and $\varphi_i(Y)=1\otimes\varphi_i$ on $U\times U$. Since $U\to\A^n$ is
\'etale, on the open $\Omega:=(U\times U)\smallsetminus\big((U\times_{\A^n}
U)\smallsetminus\Delta_U\big)$, which contains $Z\times Z$, the ideal of
the diagonal is generated by the coordinate differences
$\delta_j:=X_j-Y_j$, which form a regular sequence. Each
$\varphi_i(X)-\varphi_i(Y)$ vanishes on the diagonal, so it can be written
\[
\varphi_i(X)-\varphi_i(Y)\ =\ \sum_{j=1}^n\delta_j\,H_{ij},\qquad
H_{ij}\in\mathcal O(\Omega),
\]
a divided-difference expansion: for $U=\A^n$ and polynomial $\varphi_i$
one may take the telescoping quotients
$H_{ij}=\big(\varphi_i(Y_1,\dots,Y_{j-1},X_j,\dots,X_n)-\varphi_i(Y_1,
\dots,Y_j,X_{j+1},\dots,X_n)\big)/(X_j-Y_j)$. The
$n\times n$ matrix $H=(H_{ij})$ has entries in $\mathcal O(\Omega)$, and
its determinant $\det(H)\in\mathcal O(\Omega)$ reduces modulo the ideal
$(\varphi\otimes1,1\otimes\varphi)$ to an element of $\mathcal O(Z\times
Z)=A\otimes_kA$; this image
\[
\Delta:=\det(H)\bmod(\varphi\otimes1,\,1\otimes\varphi)\ \in\ A\otimes_kA
\]
is the \emph{B\'ezoutian} of $c$. (For $n=1$, $U=\A^1$ and $\varphi=f$ a
polynomial, $H$ is the single entry $(f(X)-f(Y))/(X-Y)$, the classical
B\'ezoutian.) Two properties follow \cite[Satz (3.1)]{SchejaStorch}. First, $\Delta$ annihilates the
$\delta_j$: in $A\otimes_kA$ we have $\varphi\otimes1=0=1\otimes\varphi$,
so the expansion says that $H$ kills the column vector $(\delta_j)_j$,
and Cramer's rule gives $\det(H)\,\delta_j=0$; since
$Z\hookrightarrow\A^n$ is a closed immersion the $\delta_j$ generate the
kernel of the multiplication $\mu\colon A\otimes_kA\to A$, so in fact
$\Delta\cdot\ker(\mu)=0$. Second, $\Delta$ is independent of the choice
of the $H_{ij}$: two choices differ, row by row, by a syzygy of the
regular sequence $\delta_1,\dots,\delta_n$, hence by Koszul exactness by a
combination of the trivial syzygies $\delta_be_a-\delta_ae_b$; and a
determinant in which a row of $H$ has been replaced by $\delta_be_a-
\delta_ae_b$ lies in the ideal generated by the $\varphi_i(X)-
\varphi_i(Y)$ (expand it along columns $a$ and $b$ and complete each
$\delta_aH_{la}+\delta_bH_{lb}$ to $\varphi_l(X)-\varphi_l(Y)$; the
completing terms are determinants with a repeated column), so the two
determinants agree in $A\otimes_kA$ \cite{BMP}. Finally
$\Delta$ is additive over clopen decompositions $Z=Z'\sqcup Z''$: for the
idempotent $e'$ of $Z'$, $e'\otimes1-1\otimes e'\in\ker\mu$ kills
$\Delta$, so $\Delta$ lies in $A'\otimes A'\oplus A''\otimes A''$, with
components the B\'ezoutians of the two pieces.

\begin{inputt}[Scheja--Storch duality {\cite[Satz (3.3)]{SchejaStorch}}]\label{inp:ss}
Let $B$ be noetherian, $P$ a smooth $B$-algebra \'etale over
$B[X_1,\dots,X_n]$, and $A=P/(\varphi_1,\dots,\varphi_n)$ finite flat over
$B$, with B\'ezoutian $\Delta\in A\otimes_BA$. There is a unique $B$-linear
$\eta\colon A\to B$ with $(\eta\otimes\mathrm{id})(\Delta)=1$, and
$\beta(a,b):=\eta(ab)$ is a nondegenerate symmetric bilinear form on the
projective $B$-module $A$, compatible with base change in $B$.
\end{inputt}

Scheja--Storch state this for polynomial $P$ \cite[\S\S2--3]{SchejaStorch}:
$\eta$ is the element $\Theta^{-1}(1)$ of their Satz (3.3), the form
$\langle b\,|\,c\rangle=\eta(bc)$ and its compatibility with base change
are on their pp.\ 182--184. Their argument uses only that the $\delta_j$
generate the diagonal ideal and form a regular sequence, which holds on
$\Omega$. Put
$w(c):=[\beta_c]\in\GW(k)$.

\begin{theorem}\label{thm:w}
The class $w(c)$ depends only on the class of $c$ in $\pfr(k)$, and $w$ is
a ring homomorphism $\pfr(k)\to\GW(k)$.
\end{theorem}

\begin{proof}[Sketch]
Additivity, stability under $\sigma$ (which adds a $1$ on the diagonal of
$H$), multiplicativity ($H$ is block diagonal for an external product) and
independence of the neighborhood are formal. For homotopy invariance let
$(\widetilde Z,W,\Phi)\in\Fr_n(\A^1)$; the point is that
$\mathcal O(\widetilde Z)$ is a \emph{free} $k[t]$-module, which the
definition does not assume. At a point $z$ of $\widetilde Z$ the $n$
equations $\Phi$ cut out, in the regular local ring $\mathcal O_{W,z}$ of
dimension $n+1$, a quotient of dimension $\ge1$ by Krull's height theorem
and $\le1$ because $\widetilde Z$ is finite over $\A^1$; so $(\Phi)$ has
height $n$, and $n$ generators of a height-$n$ ideal in a regular local
ring form a regular sequence with Cohen--Macaulay quotient
\cite[Thm.\ 17.4(iii), Thm.\ 17.3(ii)]{Matsumura}. For the one-dimensional
ring $\mathcal O_{\widetilde Z,z}$, Cohen--Macaulay means that it has no
embedded points \cite[Thm.\ 17.3(i)]{Matsumura}, and over the principal
ideal domain $k[t]$ that is all
flatness asks: a torsion element of $\mathcal O(\widetilde Z)$ would be a
function supported over finitely many values of $t$, i.e.\ an embedded
point. (This is the one-dimensional case of miracle flatness
\cite[Thm.\ 23.1]{Matsumura}, and the place where the complete-intersection
structure of the support is used: the fat family $V(x^2,xt)\subset\A^2$
has the embedded point $(x,t)$ and is not flat.) Being finitely generated
and torsion-free over $k[t]$, $\mathcal O(\widetilde Z)$ is free.
\cref{inp:ss} over $B=k[t]$ gives a nondegenerate form over
$k[t]$ specializing at $0,1$ to $w$ of the two fibers; by Harder's theorem
--- every nondegenerate symmetric bilinear form on a free $k[t]$-module is
extended from $k$ \cite{Knebusch70}, \cite[Ch.\ VII, Thm.\ 3.13]{LamSerre} --- the
two specializations agree.
\end{proof}

\begin{proposition}\label{prop:wpt}
If $c=(Z,U,\varphi)$ has \'etale support, $A=\prod_iL_i$, and
$J=\det d\varphi|_Z$ has components $J_i$, then $w(c)=\perp_i\tr_{L_i/k}
\form{J_i}$. In particular $w\circ\PT=d$ and $w(e(k,v))=\form v$, so
$\ker\PT\subseteq\ker d$ and $w$ is surjective.
\end{proposition}

\begin{proof}
The B\'ezoutian annihilates $\ker\mu$, and for $A/k$ \'etale the annihilator of
$\ker\mu$ is generated by the separability idempotent $e_{A/k}=\sum
a_i\otimes a_i^\vee$ (dual bases for the trace form), with $e_{A/k}\xi=
e_{A/k}(\mu(\xi)\otimes1)$. As $\mu(\Delta)=\det d\varphi=J$, we get
$\Delta=e_{A/k}(J\otimes1)$, and the defining property of $\eta$ forces
$\eta(a)=\tr_{A/k}(a/J)$. So $\beta=\tr_{A/k}\form{J^{-1}}\cong
\tr_{A/k}\form J$. This is Scheja--Storch's trace formula $\mathrm{Sp}=J\cdot\eta$
\cite[Satz (4.2) and Kor.\ (4.5)]{SchejaStorch}, cf.\
\cite{KassWickelgren19,BBMMO}.
\end{proof}

\section{Transversality: the Position Lemma}\label{sec:transversality}

\begin{theorem}[Position Lemma]\label{thm:position}
Every class in $\pfr(k)$ is a $\Z$-linear combination of classes with
\'etale support: $\pfr(k)=E$, and $\PT$ is surjective. The first assertion
holds over every field.
\end{theorem}

The theorem has content because a transverse preimage in topology is a
manifold, whereas a framed zero-dimensional scheme may be non-reduced.
What the geometry hands us is the ``fat'' Burnside ring of framed finite
syntomic schemes, and the theorem says that the fatness carries no
information in $\pi_0$. There is also no notion of small perturbation to
fall back on. An $\A^1$-homotopy is a family over the whole affine line
whose support is finite over $\A^1$, and the zeros of a general
perturbation of a framing escape to infinity at special values of the
parameter (\cref{fig:retop}). The statement itself is not new: in the
equational model it is Neshitov's \cite[Lemma 4.11]{Neshitov18} in
characteristic zero and Druzhinin--Kylling's \cite[Lemma 6.4]{DK} over
perfect fields of any characteristic (both go on to reach reduced rational
support, \cite[Lemma 5.4]{Neshitov18}, \cite[Lemma 6.9]{DK}), and over any
field it follows from the moving lemma of \cite[Prop.\ 2.3]{BEHKSY}
(\cref{sec:moving}).
The proof given here is elementary, and we think it is the natural one.
Inside a complete-intersection ideal, Nakayama's lemma allows one to
prescribe the top-degree forms of a framing; a framing whose top forms are
$(x_1^M,\dots,x_n^M)$ is a finite flat map; and so its deformation to
\'etale position is proper over $\A^1$ by construction.

The complete-intersection structure of the support is essential. In the
unframed Hilbert scheme every punctual scheme with rational support
degenerates, through a Rees family, to a square-zero scheme and then to
reduced points, and this collapses $\pi_0$ to the rank --- it is how
$\mathrm{Hilb}(\A^\infty)$ computes $K$-theory \cite{HJNTY}. Since
$\operatorname{Spec}(k\oplus W)$ is a complete intersection only when
$\dim W\le1$, the Rees family leaves the framed world, and any proof of
\cref{thm:position} must use the framing, on pain of proving
$\pi_0(\sph_k)(k)=\Z$. In our proof the framing carries the weight in
\cref{lem:retop}(i), where Nakayama's lemma needs $n$ generators of $I_p$;
for a $\mathfrak p$-primary ideal in an $n$-dimensional regular local ring
that is equivalent to being a complete intersection, and it already fails
for $\mathfrak m^2\subset k[x,y]$.

\subsection{Re-topping}\label{sec:retop}
The proof has three steps: First we reduce to a framing given by polynomials at a single
closed point, so that there is a well-defined notion of the degree of each
equation. Second --- the heart of the matter --- we change those
polynomials so that their top-degree parts are $x_1^M,\dots,x_n^M$,
without changing the support or the class. Third we deform the new
framing by a constant vector, $\psi-t\gamma$; because the top-degree
parts are what control the behavior at infinity and they no longer
depend on $t$, no zero can escape, the family is finite over the
$t$-line, and for a generic $\gamma$ its fiber at $t=1$ is \'etale. The
reader may keep the one-variable picture of \cref{fig:retop} in mind
throughout: a monic polynomial of degree $M$ has $M$ roots for every value
of its constant term, whereas a polynomial whose leading term is allowed to
vary with $t$ loses roots to infinity exactly when that term vanishes.

\emph{Step one: polynomial framings.} By additivity and \cref{thm:rigid},
every class is a combination of classes $[(Z,V,\varphi)]$ with $Z$
punctual at a closed point $p\in\A^n$, $V\ni p$ a Zariski open subset,
and $\varphi$ a system of $n$ polynomials generating the ideal
$I_p=I(Z)_p$ of $Z$ at $p$. Indeed, for a punctual framed correspondence
$(Z,U,\varphi^U)$ the ideal of $Z$ in $\mathcal O_{U,z}$ is extended from
$I_p$ by \'etale descent, so $I_p$ needs at most $n$ generators; as
$\mathrm{ht}\,I_p=n$, it needs exactly $n$, and any $n$ generators, which
we may take in $I(Z)$, form a regular sequence. Then \cref{thm:rigid}
compares the given framing with the polynomial one up to a twist by a unit
of $\kappa(p)$, which is realized by multiplying $\varphi_1$ by a
polynomial. (This step uses neither infiniteness nor perfectness, and $p$
need not be rational.)

\emph{Why the framing may be changed at all.} If $\psi_j=\varphi_j+(\text{an element of }I^2)$ for each
$j$, then $\psi$ still generates $I$ at every point $p$ of $Z$: since
$I_p\subseteq\mathfrak p_p$ we have $I_p^2\subseteq\mathfrak p_pI_p$, so
\[
I_p=(\varphi)\ \subseteq\ (\psi)+I_p^2\ \subseteq\ (\psi)+\mathfrak p_pI_p
\ \subseteq\ I_p,
\]
and Nakayama's lemma gives $I_p=(\psi)$. Moreover $\psi$ and $\varphi$
induce the same trivialization of $I/I^2$, so by \cref{thm:rigid} the
framed correspondences $(Z,-,\varphi)$ and $(Z,-,\psi)$ have the same
class (their transition matrix is congruent to the identity modulo
$\mathfrak p_p$). In short: \emph{modifying a framing by elements of
$I^2$ is free.} This is also where the hypothesis that $Z$ be a complete
intersection is used: Nakayama needs $I_p$ to be generated by the $n$
elements $\varphi_j$, which for a punctual ideal in the $n$-dimensional
regular local ring $\mathcal O_{\A^n,p}$ is exactly the complete
intersection condition.

\emph{Step two: what the freedom buys.} The ideal $I^2$ has finite
colength, so for $M$ large every monomial, in particular $x_j^M$, is
congruent modulo $I^2$ to a polynomial $g_j$ of degree $<M$. Then
$x_j^M-g_j$ is an element of $I^2$ whose top-degree part is exactly
$x_j^M$, and adding it to $\varphi_j$ puts a prescribed top on the $j$-th
equation at no cost. That is the whole idea; the lemma below records it
together with two refinements, a genericity statement (iv) that keeps the
new zeros away from the old ones in good position, and the choice of $M$
prime to the characteristic, which keeps the Jacobian honest.

\begin{lemma}[Re-topping]\label{lem:retop}
Let $Z\subset\A^n_k$ be finite with ideal $I$, generated by polynomials
$\pi_1,\dots,\pi_r$, and let $\varphi_1,\dots,\varphi_n\in I$ generate $I$
locally along $Z$. Put $D:=2\max_l\deg\pi_l+1$ and
$H:=I^2\cap k[x]_{\le D}$, a finite-dimensional space. Choose an integer
$M$ such that
\[
M>\max\big(D,\ \deg\varphi_1,\dots,\deg\varphi_n\big),\qquad
\operatorname{char}k\nmid M,\qquad
k[x]_{<M}\twoheadrightarrow k[x]/I^2 ,
\]
and polynomials $g_j\in k[x]_{<M}$ with $g_j\equiv x_j^M\bmod I^2$. For
$h=(h_1,\dots,h_n)\in H^n$ set
\[
\psi_j\ :=\ \varphi_j\ +\ (x_j^M-g_j)\ +\ h_j .
\]
Then:
\begin{enumerate}
\item[\textup{(i)}] $\psi\equiv\varphi\bmod I^2$, and $(\psi)\mathcal
O_{\A^n,p}=I_p$ for every $p\in Z$; so $V(\psi)$ agrees with $Z$ near
$Z$, and $\psi$ induces the same trivialization of $I/I^2$ as $\varphi$.
\item[\textup{(ii)}] Each $\psi_j$ has degree exactly $M$, with top-degree
part $x_j^M$.
\item[\textup{(iii)}] The top-degree part of $\det d\psi$ is
$M^n(x_1\cdots x_n)^{M-1}$, which is nonzero; in particular
$\det d\psi\not\equiv0$.
\item[\textup{(iv)}] For $h$ outside a proper closed subset of $H^n$, the
scheme $V(\psi)\smallsetminus Z$ is \'etale over $k$.
\end{enumerate}
\end{lemma}

\begin{proof}
(i) Both $x_j^M-g_j$ and $h_j$ lie in $I^2$, so this is the observation
above. (ii) Since $\deg\varphi_j$, $\deg g_j$ and $\deg h_j\le D$ are all
less than $M$, the only term of degree $M$ in $\psi_j$ is $x_j^M$.
(iii) Differentiating, $d\psi=M\,\mathrm{diag}(x_1^{M-1},\dots,x_n^{M-1})
+(\text{terms of degree}\le M-2)$, so the degree-$n(M-1)$ part of the
determinant is $M^n(x_1\cdots x_n)^{M-1}$, nonzero because
$\operatorname{char}k\nmid M$. (iv) This is a dimension count, and the
role of $D$ is to make it work. Fix $x\in\A^n(\bar k)$ not in $Z$; then
some generator has $\pi_l(x)=a\ne0$. For an affine-linear $q$ the
polynomial $\pi_l^2q$ lies in $I^2$ and has degree $\le D$, hence lies in
$H$, and its value and gradient at $x$ are $a^2q(x)$ and
$2aq(x)\nabla\pi_l(x)+a^2\nabla q(x)$; choosing $q$ we can prescribe both
freely. So the affine-linear map $H\to\bar k\times\bar k^n$,
$h\mapsto(h(x),\nabla h(x))$, is onto, and applying this in each
coordinate, the conditions $\psi(x)=0$ and $\det d\psi_x=0$ cut out a
subset of $H^n$ of codimension $\ge n+1$. Letting $x$ vary over the
$n$-dimensional $\A^n\smallsetminus Z$, the bad pairs $(x,h)$ form a set
of dimension $<\dim H^n$, whose projection to $H^n$ is therefore not
dense; as $k$ is infinite, a $k$-point $h$ of $H^n$ avoids its closure,
and for that $h$ every point of $V(\psi)$ off $Z$ is a simple zero, i.e.\
$V(\psi)\smallsetminus Z$ is \'etale. Since $\deg h_j\le D<M$, this choice
disturbs none of (i)--(iii).
\end{proof}

\emph{The choices in the lemma.} The lemma makes two choices, and the
conditions on each are matched to the conclusions they serve. The three
conditions on the degree $M$ do three different things: $M$ exceeds the
degrees of the $\varphi_j$ and of the correction terms so that $x_j^M$
really is the top-degree part of $\psi_j$, which is (ii); $M$ is large
enough that $x_j^M$ can be reduced modulo $I^2$ to a polynomial of lower
degree, which is what produces an element of $I^2$ with that top; and
$M$ is prime to the characteristic so that the Jacobian does not vanish
identically, which is (iii) --- in characteristic $p$ the system
$(x_1^p,\dots,x_n^p)$ has all the finiteness one could want and not a
single \'etale fiber. The nudge $h$ is a generic element of a
finite-dimensional space of low-degree elements of $I^2$. It is invisible
to (i)--(iii), and its only purpose is (iv): the re-topped system has
$M^n$ zeros in all, of which $\deg Z$ lie on $Z$, and $h$ makes the
others honest reduced points so that they can be subtracted off. One
could replace $x_j^M$ by any system of top forms $\sum_ic_{ji}x_i^M$ with
$(c_{ji})$ invertible, or indeed by any forms of degree $M$ with no
common zero but the origin and nonvanishing Jacobian; nothing is gained
by the generality.

\emph{Step three: finiteness and the deformation.} A system $\psi$ of
$n$ polynomials of degree $M$ whose top-degree parts $x_1^M,\dots,x_n^M$
have no common zero but the origin is a finite flat morphism
$\A^n\to\A^n$ of degree $M^n$. To see this,
homogenize: the closure in $\PP^n$ of a fiber $V(\psi-c)$ is
$V(\bar\psi_j-c_jX_0^M)$, which meets the hyperplane at infinity
$\{X_0=0\}$ in the common zeros of the top forms, i.e.\ nowhere; so each
fiber is a closed subscheme of $\PP^n$ contained in affine space, hence
proper and affine, hence finite. The same argument over the target shows
that $\psi$ is proper, so finite, and flatness is miracle flatness. The
point of the construction is that this applies uniformly in a family:
the top-degree parts of $\psi_j-t\gamma_j$ over $k[t]$ are still
$x_j^M$, because $\gamma$ is constant, so the subscheme
$V(\psi-t\gamma)$ of $\A^n\times\A^1_t$ is finite flat over the whole
$t$-line. That is exactly
the condition an $\A^1$-homotopy must satisfy, and it has been arranged
in advance rather than checked after the fact. The third and last
choice is the vector $\gamma\in k^n$, a generic point of the target,
whose job is to move the fiber of $\psi$ off the ramification locus: if
$\det d\psi\not\equiv0$ then $V(\det d\psi)$ is a proper closed subset of
the source, its image under the finite map $\psi$ is a proper closed
subset of the target, and every $\gamma$ off that image has \'etale
fiber $\psi^{-1}(\gamma)$; such a $k$-point $\gamma$ exists because $k$
is infinite. \Cref{fig:retop} shows the simplest instance.

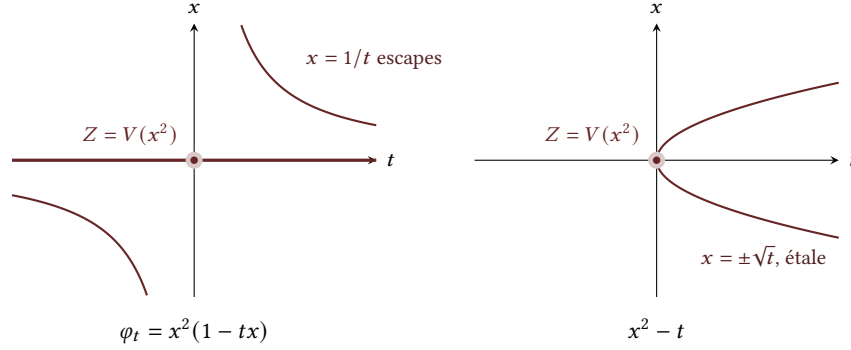
\begin{figure}[t]
\centering
\begin{tikzpicture}[every node/.style={font=\scriptsize}]
\begin{axis}[name=left,width=6.4cm,height=5.2cm,axis lines=middle,
  xlabel={$t$},ylabel={$x$},xlabel style={anchor=west},ylabel style={anchor=south},
  xmin=-1.7,xmax=1.7,ymin=-2.3,ymax=2.3,xtick=\empty,ytick=\empty,clip=false]
\node[anchor=north,font=\footnotesize] at (axis description cs:0.5,-0.04) {$\varphi_t=x^2(1-tx)$};
\addplot[dark-red,very thick,domain=-1.7:1.7,samples=2] {0};
\addplot[dark-red,thick,domain=0.44:1.7,samples=100] {1/x};
\addplot[dark-red,thick,domain=-1.7:-0.44,samples=100] {1/x};
\fill[dark-red!25] (axis cs:0,0) circle (3.2pt);
\fill[dark-red] (axis cs:0,0) circle (1.4pt);
\node[dark-red,anchor=south east] at (axis cs:-0.05,0.08) {$Z=V(x^2)$};
\node[dark-red!80!black,anchor=west] at (axis cs:0.95,1.7) {$x=1/t$ escapes};
\end{axis}
\begin{axis}[at={(left.east)},anchor=west,xshift=1.3cm,width=6.4cm,height=5.2cm,
  axis lines=middle,xlabel={$t$},ylabel={$x$},xlabel style={anchor=west},
  ylabel style={anchor=south},
  xmin=-1.7,xmax=1.7,ymin=-2.3,ymax=2.3,xtick=\empty,ytick=\empty,clip=false]
\node[anchor=north,font=\footnotesize] at (axis description cs:0.5,-0.04) {$x^2-t$};
\addplot[dark-red,thick,domain=0:1.7,samples=100] {sqrt(x)};
\addplot[dark-red,thick,domain=0:1.7,samples=100] {-sqrt(x)};
\fill[dark-red!25] (axis cs:0,0) circle (3.2pt);
\fill[dark-red] (axis cs:0,0) circle (1.4pt);
\node[dark-red,anchor=south east] at (axis cs:-0.05,0.08) {$Z=V(x^2)$};
\node[dark-red!80!black,anchor=north] at (axis cs:1.0,-1.25) {$x=\pm\sqrt t$, \'etale};
\end{axis}
\end{tikzpicture}
\caption{Transversality without small perturbations, in the simplest case
$n=1$, $Z=V(x^2)$, $\varphi=x^2$ (drawn over $\R$). Left: the perturbation
$\varphi_t=x^2-tx^3$ has, for $t\ne0$, the extra zero $x=1/t$, which runs
off to infinity as $t\to0$; the support, the double line $x=0$ together
with the hyperbola $tx=1$, is not finite over the $t$-line, so this is not
an $\A^1$-homotopy. Right: $\varphi=x^2$ already has the shape required by
\cref{lem:retop}, and the family $x^2-t$ has support finite flat of degree
$2$ over the whole $t$-line, with fiber the fat point $Z$ at $t=0$ and an
\'etale fiber for every $t\ne0$.}
\label{fig:retop}
\end{figure}

\begin{proof}[Proof of \cref{thm:position}]
By step one, take $[(Z,V,\varphi)]$ punctual at $p$ with $\varphi$
polynomial generating $I_p$, and apply \cref{lem:retop} with a generic
$h$ to obtain $\psi$. By (ii), (iii) and step three there is a $\gamma$
for which $(V(\psi-t\gamma),\A^n\times\A^1,\psi-t\gamma)$ is an
$\A^1$-homotopy --- its support is finite over $\A^1$ --- from
$(V(\psi),\A^n,\psi)$ to a correspondence with \'etale support; hence
$[V(\psi),\A^n,\psi]\in E$. By (i) and (iv), $V(\psi)=Z\sqcup Y$ with $Y$
\'etale, so additivity gives $[Z,\A^n\smallsetminus Y,\psi]=
[V(\psi),\A^n,\psi]-[Y,\A^n\smallsetminus Z,\psi]\in E$. Finally, by (i)
and \cref{thm:rigid}, $[Z,\A^n\smallsetminus Y,\psi]=[Z,V,\varphi]$: the
transition matrix between $\varphi$ and $\psi$ is congruent to the
identity modulo $\mathfrak p_p$, so its determinant is $1$ in
$\kappa(p)^\times$. Surjectivity of $\PT$ is \cref{prop:ptring}.
\end{proof}

\subsection{The moving-lemma route, and the density of the \'etale locus}
\label{sec:moving}
Over an arbitrary field the first assertion of \cref{thm:position}
follows from \cite[Prop.\ 2.3]{BEHKSY}, a moving lemma whose idea is
Levine's: a framed correspondence $(Z\subset\A^n,\varphi)$ of degree $d$ is
a $k$-point $x$ of the framed Hilbert scheme $X=\Hilbfr_d(\A^n_k)$ of
\cite[Def.\ 5.1.7]{EHKSY} --- finite syntomic subschemes of $\A^n$ with
trivialized conormal sheaf --- which is smooth of dimension $dn(n+1)$
\cite[\S5.1, following Def.\ 5.1.7]{EHKSY};
the moving lemma applied to $x$ and the dense open \'etale locus
$X^{\text{\'et}}$ gives tangentially framed correspondences $\alpha,\beta$
into $X^{\text{\'et}}$ with finite \'etale supports and a homotopy
$x+\alpha\sim\beta$; composing with the universal family and transporting
back to the equational model by the comparison equivalences
$L_{\A^1}h^{\mathrm{efr}}\simeq L_{\A^1}h^{\mathrm{nfr}}\simeq
L_{\A^1}h^{\mathrm{fr}}$ on affines \cite[Cor.\ 2.2.20, 2.3.25]{EHKSY}
gives $[c]\in E$. That route uses the comparison of models, the smoothness
of the framed Hilbert scheme and Bertini theorems over finite fields
\cite{Poonen,CharlesPoonen}; the proof above uses none of these, at the
price of $k$ infinite. The one ingredient of the moving route for which we
know no reference is the density of $X^{\text{\'et}}$, which \cite{BEHKSY}
use without comment; re-topping proves it.

\begin{proposition}\label{prop:dense}
For any field $k$ and $d,m\ge0$ the finite \'etale locus is dense in
$\Hilbfr_d(\A^m_k)$.
\end{proposition}

\begin{proof}[Sketch]
For $k$ infinite and $x=(Z,\tau)\in X(k)$, lift $\tau$ to $\varphi_j\in I$
generating $I$ along $Z$, re-top to $\psi$, and choose $\gamma$ with
$\psi^{-1}(\gamma)$ \'etale. The family $V(\psi-t\gamma)\subset\A^m\times
\A^1$ is finite flat over $\A^1$ with conormal sheaf trivialized by
$\psi-t\gamma$, of degree $M^m$, with fiber $Z\sqcup Y$ over $0$ and
\'etale generic fiber. Over the henselization $S$ of $\A^1$ at $0$ it
splits into local pieces; the union $W$ of those over the points of $Z$ is
finite flat of degree $d$ over $S$ with closed fiber $(Z,\tau)$, i.e.\ a
morphism $S\to X$ sending the closed point to $x$ and the generic point into
$X^{\text{\'et}}$. So $x\in\overline{X^{\text{\'et}}}$. For $k$ arbitrary,
a closed point of the open $X\smallsetminus\overline{X^{\text{\'et}}}$ would
lift to a $\kappa(s)$-point of the base change to the infinite field
$\kappa(s)$, contradicting the first part.
\end{proof}

\section{Marks: \texorpdfstring{$\PT$}{PT} factors through the Dress map}\label{sec:marks}

Segal's step (S5) composes $\PT$ with a computable invariant and invokes
Burnside's theorem. Our composite is $d$ (\cref{prop:wpt}); what
must be shown is the converse $\ker d\subseteq\ker\PT$, and we prove the
sharper statement that a decorated \'etale scheme is determined, as a
framed class, by its scaled trace form.

\begin{theorem}\label{thm:marks}
There is a ring homomorphism $s\colon\GW(k)\to\pfr(k)$ with
$s\form v=e(k,v)$ such that $w\circ s=\mathrm{id}$ and $\PT=s\circ d$;
that is, in the diagram
\[
\begin{tikzcd}[row sep=large, column sep=large]
\Au(\Gamma_k) \arrow[r, "\PT"] \arrow[dr, "d"'] & \pfr(k)
\arrow[d, "w", shift left=0.7ex] \\
& \GW(k) \arrow[u, "s", shift left=0.7ex]
\end{tikzcd}
\]
the triangle $w\circ\PT=d$ of \cref{prop:wpt} is complemented by
$\PT=s\circ d$, and $s$ is a section of $w$. In particular
$\ker d\subseteq\ker\PT$ and $E=\operatorname{im}(s)$.
\end{theorem}

Everything in this section takes place on the affine line: the
homotopies are pencils of pointed rational functions, that is, level-one
framed correspondences whose support is the zero set of a monic polynomial
in one variable, and the invariants are those of maps $\PP^1\to\PP^1$. A
\emph{pointed rational function of degree $d$} is
$F=A/B$ with $A$ monic of degree $d$, $\deg B<d$ and $\res(F):=\res(A,B)=
\prod_{A(\rho)=0}B(\rho)\ne0$; these form an open $\mathcal F_d\subseteq
\A^{2d}$, and two are \emph{naively homotopic} if connected by a chain of
maps $\A^1\to\mathcal F_d$. To $F$ corresponds the level-one framed
correspondence $c(F)=(V(A),\A^1\smallsetminus V(B),A/B)$, and a naive
homotopy is an $\A^1$-homotopy, because $A_t$ monic makes the support
finite over $\A^1$ and $\res(A_t,B_t)\in k[t]^\times=k^\times$ keeps $B_t$
invertible on it. If $A$ is separable with $V(A)=\coprod\operatorname{Spec}
L_i$, $L_i=k(\theta_i)$, then
\begin{equation}\label{eq:cF}
[c(F)]=\sum_ie\big(L_i,A'(\theta_i)/B(\theta_i)\big)\in E,
\end{equation}
the framing determinant of $A/B$ at $\theta_i$ being $A'(\theta_i)/
B(\theta_i)$. The \emph{B\'ezout form} $\Bez(F)$ is the symmetric form on
$k^d$ with Gram matrix $(b_{ij})$, where $\big(A(X)B(Y)-A(Y)B(X)\big)/(X-Y)=
\sum b_{ij}X^iY^j$.

\begin{inputt}[Cazanave {\cite[Thm.\ 3.6, Cor.\ 3.10]{Cazanave}}]\label{inp:caz}
Two pointed rational functions of the same degree are naively homotopic if
their B\'ezout forms are isometric and their resultants are equal in
$k^\times$.
\end{inputt}

Cazanave proves that $F\mapsto(\Bez(F),\res(F))$ classifies naive
homotopy classes; we use only this direction, and neither his monoid
structure nor his theorem (via Morel) that the naive classes
group-complete to $[\PP^1,\PP^1]^{\A^1}$. The link to \cref{sec:w} is
that the B\'ezout form is a trace form: for $A$ separable with roots
$\rho_i$, Vandermonde $W=(\rho_i^{\,j})$ and $D=\mathrm{diag}
(A'(\rho_i)B(\rho_i))$, evaluating the B\'ezoutian at pairs of roots gives
$WbW^{\top}=D$, while the Gram matrix of $\tr_{L/k}\form{A'B}$ on
$L=k[x]/(A)$ in the power basis is $G=W^{\top}DW$; so $G=SbS$ with
$S=W^{\top}W$ the invertible Gram matrix of $\tr_{L/k}\form1$, and
\begin{equation}\label{eq:bezw}
\Bez(F)\cong\tr_{L/k}\form{A'B}\cong\perp_i\tr_{L_i/k}\form{A'(\theta_i)
/B(\theta_i)}=w\big(c(F)\big)
\end{equation}
by \cref{prop:wpt} (the separable case of
\cite{KassWickelgren20}). Taking determinants, $\det G=\disc(A)^2\det b$ and
$\det G=\disc(A)\res(A,A')\res(A,B)$, whence
$\det\Bez(F)=(-1)^{d(d-1)/2}\res(A,B)$; so \emph{two pointed rational
functions of the same degree with separable numerators and isometric
B\'ezout forms have resultants that agree modulo squares.}

For distinct $a_1,\dots,a_d\in k$ and $v_i\in k^\times$ the \emph{split
configuration} $G(a;v)=A/B$ has $A=\prod(x-a_i)$ and $B$ the Lagrange
interpolant, the unique polynomial of degree $<d$ with the prescribed
values $B(a_i)=A'(a_i)/v_i$ (this needs $|k|\ge d$, with $d$ unbounded
below); by \eqref{eq:cF} and \eqref{eq:bezw},
\[
[c(G(a;v))]=\sum_ie(k,v_i),\quad\Bez(G(a;v))\cong\form{v_1,\dots,v_d},
\quad\res(G(a;v))=\prod_iA'(a_i)/v_i ,
\]
and replacing $v_1$ by $v_1t^{-2}$ keeps the form and multiplies the
resultant by $t^2$.

\begin{proposition}\label{prop:relations}
In $\pfr(k)$, for all $v,w,t\in k^\times$: \textup{(a)} $e(k,vt^2)=e(k,v)$;
\textup{(b)} $e(k,v)+e(k,w)=e(k,v+w)+e(k,vw(v+w))$ if $v+w\ne0$;
\textup{(c)} $e(k,v)+e(k,-v)=e(k,1)+e(k,-1)$. Hence $\form v\mapsto e(k,v)$
defines a ring homomorphism $s\colon\GW(k)\to\pfr(k)$ with $w\circ s=
\mathrm{id}$.
\end{proposition}

\begin{proof}
The strategy for each relation is the same: write both sides as the
classes of degree-two split configurations $G$ and $G'$, check that
$\Bez(G)\cong\Bez(G')$ and $\res(G)=\res(G')$, and conclude from
\cref{inp:caz} that $G$ and $G'$ are naively homotopic, so that
$[c(G)]=[c(G')]$. For a degree-two configuration with points $a_1,a_2$
and units $v_1,v_2$ one has $A'(a_1)=a_1-a_2=-A'(a_2)$, so its resultant
is $-(a_1-a_2)^2/(v_1v_2)$; the B\'ezout form is $\form{v_1,v_2}$
regardless of the points, and the resultant can be adjusted at will by
moving the points.

(a) Let $G$ have units $(v,1)$ and $a_1-a_2=1$, and $G'$ units $(vt^2,1)$
and $a_1-a_2=t$. The forms $\form{v,1}$ and $\form{vt^2,1}$ are
isometric, and both resultants equal $-1/v$. Hence
$e(k,v)+e(k,1)=e(k,vt^2)+e(k,1)$, and $e(k,1)$ cancels.

(b) Let $G$ have units $(v,w)$ and separation $\delta$, and $G'$ units
$(v+w,\,vw(v+w))$ and separation $\delta(v+w)$. The form $\form{v,w}$
represents $v+w$ and has determinant $vw$, so it is isometric to
$\form{v+w,\,vw(v+w)}$; and both resultants equal $-\delta^2/(vw)$.

(c) Let $G$ have units $(v,-v)$ and separation $\delta$, and $G'$ units
$(1,-1)$ and separation $\delta/v$. The form $\form{v,-v}$ is isotropic
with determinant $-1$, hence isometric to $\form{1,-1}$; and both
resultants equal $\delta^2/v^2$.

Relations (a)--(c) are exactly Witt's presentation of $\GW(k)$
\cite[Satz 7]{Witt37}, \cite[Thm.\ I.5.2]{Lam05}, so $s$ is a well-defined additive
map; it is multiplicative by \cref{prop:ptring}, and $w\circ s=\mathrm{id}$
on generators by \cref{prop:wpt}.
\end{proof}

\begin{proposition}\label{prop:etalesplit}
For every finite separable $L/k$ and $u\in L^\times$,
$e(L,u)=s(\tr_{L/k}\form u)$.
\end{proposition}

\begin{proof}
Let $f$ be the minimal polynomial of a primitive element $\theta$ and
$B$ the unique polynomial of degree $<d$ with $B(\theta)=f'(\theta)/u$ in
$L$, so that $F=f/B$ has
$[c(F)]=e(L,u)$ and $\Bez(F)\cong\tr_{L/k}\form u=:q$ by \eqref{eq:cF},
\eqref{eq:bezw}. Diagonalize $q\cong\form{v_1,\dots,v_d}$ and take a split
configuration $G_0$ with these units; its resultant agrees with $\res(F)$
modulo squares, and rescaling $v_1$ makes them equal. By
\cref{inp:caz}, $F$ and the rescaled $G$ are naively homotopic, so
$e(L,u)=[c(G)]=\sum_ie(k,v_i)=s(q)$, using
\cref{prop:relations}(a).
\end{proof}

This is the analogue of Burnside's theorem: every orbit is broken into
rational points once the quadratic invariant is recorded, by a pencil of
pointed rational functions, finite over $\A^1$ because the numerators are
monic of fixed degree, with the resultant keeping track of the framing.
\cref{thm:marks} follows, since $\PT$ and $s\circ d$ agree on
generators and $d$ is onto.

\mainthm*

\begin{proof}[Proof of \cref{thm:main}]
By \cref{thm:position,thm:marks},
$\pfr(k)=E=\operatorname{im}(s)$ and $w\circ s=\mathrm{id}$; a surjection
with a left inverse is bijective, so $w=s^{-1}$ is a ring isomorphism
$\pfr(k)\cong\GW(k)$, and \cref{inp:rec} transports it to
$\pi_0(\sph_k)(k)$. Then $\ker\PT=\ker(w\circ\PT)=\ker d$.
\end{proof}

\section{Principal quadratic divisors}\label{sec:reciprocity}

This section identifies $\ker(d)$. It is self-contained quadratic form
theory over a perfect field $k$ of characteristic $\ne2$, not necessarily
infinite, and uses from the rest of the paper only Witt's presentation and,
for interpretation, \cref{prop:wpt}. The relations are a
quadratic refinement of the fact that a principal divisor on $\PP^1$ has
degree zero.

\subsection{Reciprocity on \texorpdfstring{$\PP^1$}{P1}}
For a closed point $p$ of $\A^1_k$ let $\pi_p$ be its monic irreducible
polynomial, $\kappa(p)=k[t]/(\pi_p)$, $\theta_p$ the image of $t$, and
\[
\varepsilon_p:=\pi_p'(\theta_p)\in\kappa(p)^\times
\]
the \emph{Euler unit} of $p$; its name comes from \cref{lem:euler}, by
which $x\mapsto\tr_{\kappa(p)/k}(x/\varepsilon_p)$ is the functional dual
to $\theta_p^{\deg\pi_p-1}$. Put
$\kappa(\infty)=k$, $\pi_\infty=1/t$, and for $f\in k(t)^\times$ write
$f=\pi_p^{v_p(f)}g_p$ with $g_p$ a unit at $p$ and $w_p(f):=g_p(p)$.

\begin{lemma}[Euler's formula]\label{lem:euler}
Let $F$ be monic separable of degree $n$ and $E=k[t]/(F)$, in which the
image of the derivative $F'$ is a unit. For $z\in E$, represented by its
unique polynomial of degree $<n$, the trace $\tr_{E/k}\big(z\cdot(F')^{-1}
\big)$ is the coefficient of $t^{n-1}$ in $z$.
\end{lemma}

\begin{proof}
Lagrange interpolation: $z=\sum_\rho z(\rho)\prod_{\rho'\ne\rho}(t-\rho')
/F'(\rho)$ over the roots, and the $t^{n-1}$-coefficient of the right side
is $\sum_\rho z(\rho)/F'(\rho)=\tr_{E/k}\big(z\cdot(F')^{-1}\big)$.
\end{proof}

\begin{proposition}[Reciprocity on $\PP^1$]\label{prop:recipP1}
For $F$ monic separable of degree $n\ge1$ and $\lambda\in k^\times$,
\[
\sum_{p\in V(F)}\tr_{\kappa(p)/k}\form{\varepsilon_p\,w_p(\lambda F)}
\ =\ \lfloor n/2\rfloor\,h+[n\ \mathrm{odd}]\,\form\lambda\qquad\text{in }\GW(k).
\]
\end{proposition}

\begin{proof}
With $F=\prod_p\pi_p$ and $E=\prod\kappa(p)$ one has $F'(\theta_p)=
\varepsilon_p\prod_{q\ne p}\pi_q(\theta_p)$ and $w_p(\lambda F)=\lambda
\prod_{q\ne p}\pi_q(\theta_p)$, so $\varepsilon_pw_p(\lambda F)=\lambda
F'(\theta_p)$ and the left side is $\tr_{E/k}\form{\lambda F'}\cong
\tr_{E/k}\form{\lambda/F'}$, the form $(x,y)\mapsto\ell(\lambda xy)$ with
$\ell$ the $t^{n-1}$-coefficient (\cref{lem:euler}). In the basis
$1,t,\dots,t^{n-1}$ its Gram matrix vanishes for $i+j<n-1$ and equals
$\lambda$ for $i+j=n-1$, so the span of $1,\dots,t^{m-1}$,
$m=\lfloor n/2\rfloor$, is totally isotropic and the form is
$mh\perp q_0$ with $\rk q_0\le1$. For $n$ odd, $q_0=\form\delta$ with
$(-1)^m\delta\equiv\det=(-1)^{n(n-1)/2}\lambda^n$, and
$n(n-1)/2-m=(n-1)^2/2$ is even, so $\delta\equiv\lambda$.
\end{proof}

\begin{remark}\label{rem:scharlau}
This is the case of $\PP^1$ of Scharlau's reciprocity law
\cite[Thm.\ 4.1]{Scharlau72}, \cite[Ch.\ 6]{Scharlau85}, itself the case of $\PP^1$ of
the residue theorem of Geyer--Harder--Knebusch--Scharlau \cite{GHKS}:
for $q\in\W(k(t))$, $\sum_{p\in\A^1}s_p(\partial_pq)=\partial_\infty q$ in
$\W(k)$, where $\partial_p$ is the second residue and $s_p=\tr_{\kappa(p)
/k}(\form{\varepsilon_p}\otimes-)$ is the transfer along $x\mapsto\tr(x/
\varepsilon_p)$, the functional dual to $\theta_p^{\deg\pi_p-1}$ (Scharlau
defines $s_p$ by exactly this functional and takes $s_\infty=-\mathrm{id}$,
which accounts for the sign). The proposition lifts it to $\GW(k)$, and through
\cref{prop:wpt} its left side is the Scheja--Storch form of the
framed point $(V(F),\A^1,\lambda F)$: the degree of a framed point on the line is
determined by its behavior at infinity.
\end{remark}

\subsection{The subgroup \texorpdfstring{$R$}{R}}

\begin{definition}\label{def:R}
For $f\in k(t)^\times$ the \emph{quadratic divisor} is
\[
\rho(f):=\sum_{\substack{p\in(\A^1_k)^{(1)}\\ v_p(f)\ \mathrm{odd}}}
\big(\kappa(p),\varepsilon_pw_p(f)\big)-[v_\infty(f)\ \mathrm{odd}]
\big(k,w_\infty(f)\big)\in\Au(\Gamma_k).
\]
Its rank is even, since reducing $\sum_{\PP^1}v_p(f)[\kappa(p):k]=0$ modulo
$2$ gives it up to the term $2[v_\infty(f)\ \mathrm{odd}]$; write
$\rk\rho(f)=2r(f)$ and set $R(f):=\rho(f)-r(f)h_{\!A}$, with
$h_{\!A}=(k,1)+(k,-1)$, the \emph{principal quadratic divisor} of $f$. Let
$R\subseteq\Au(\Gamma_k)$ be the subgroup they generate.
\end{definition}

Every finite point of odd valuation enters with a plus, zeros and poles
alike, and only $\infty$ with a minus: the second residue sees the parity
of the valuation, not its sign. For $f=A/B$ as in \eqref{eq:cF}, the decoration at a zero $\theta$ of
$A$ is $\varepsilon_\theta w_\theta(f)=A'(\theta)/B(\theta)$, the framing
determinant of $c(F)$ there, and at a simple zero $\theta$ of $B$ it is
$\varepsilon_\theta w_\theta(f)\equiv B'(\theta)/A(\theta)$ modulo
squares, the framing determinant of the reciprocal function; so, up to
squares, $\rho(f)$ is the framed point $c(F)$ together with the framed
point of $1/F$ at the odd-order poles, minus the contribution at
infinity.

\begin{proposition}\label{prop:RinKer}
The subgroup $R$ lies in $\ker d$. More precisely, if $f=\lambda Fg^2$
with $\lambda\in k^\times$ and $F$ monic squarefree of degree $n$, then
$\rho(\lambda F)=\sum_{p\in V(F)}(\kappa(p),\lambda F'(\theta_p))-
[n\ \mathrm{odd}](k,\lambda)$, $r(\lambda F)=\lfloor n/2\rfloor$, and
$d(\rho(f))=d(\rho(\lambda F))$, $r(f)=r(\lambda F)$.
\end{proposition}

\begin{proof}
Factor $f=\lambda\prod\pi_p^{v_p}$, let $F$ be the product of the
$\pi_p$ with $v_p$ odd (separable as $k$ is perfect) and $g$ the obvious
cofactor. The points of odd valuation of $f$ and of $\lambda F$ coincide,
and $w_p(f)=w_p(\lambda F)w_p(g)^2$ there, including at $\infty$ where
$w_\infty(\lambda F)=\lambda$. So the decorations agree up to squares, and
$d(\rho(\lambda F))=\lfloor n/2\rfloor h$ by \cref{prop:recipP1} and
$\varepsilon_pw_p(\lambda F)=\lambda F'(\theta_p)$.
\end{proof}

\begin{example}\label{ex:sqrt2}
Over $\Q$, $f=(t+2)/(t^2-2)$ has $\rho(f)=\big(\Q(\sqrt2),4(1+\sqrt2)\big)
+(\Q,\tfrac12)-(\Q,1)$ (the pole $\Q(\sqrt2)$ with $\varepsilon=2\sqrt2$,
$w=2+\sqrt2$; the zero $-2$ with $w=1/2$; a simple zero at $\infty$ with
$w_\infty=1$) and $R(f)=\rho(f)-h_{\!A}$. Since $\tr_{\Q(\sqrt2)/\Q}
\form{1+\sqrt2}\cong\form{2,-1}$, $d(R(f))=\form{2,-1}+\form2-2\form1-
\form{-1}=2\form2-2\form1=0$, because $2x^2+2y^2=(x+y)^2+(x-y)^2$;
note that $\form{2,-1}\ne h$ in $\GW(\Q)$, since $2$ is not a square, so
the relation needs the rational zero to close. The pencil $\varphi_s=(t+2)(t^2+st+2s-2)$
on $\A^1\smallsetminus\{-2\}$, whose quadratic factor equals $2$ at $t=-2$,
is an $\A^1$-homotopy from the point $\Q(\sqrt2)$ with determinant
$4(1+\sqrt2)$ (at $s=0$) to the rational points $0,-1$ with determinants
$2,-1$ (at $s=1$): the splitting pencil of \cref{prop:etalesplit}
in the smallest case, with residue fields jumping while the trace form
$\form{2,-1}$ stays constant (\cref{fig:pencil}).
\end{example}

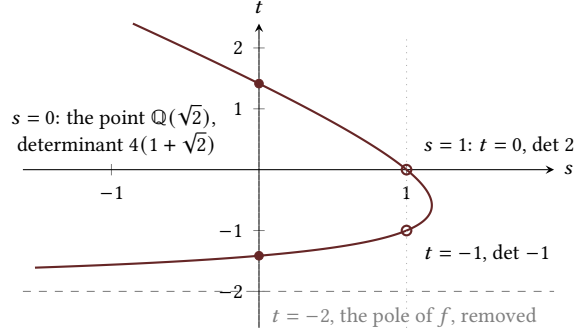
\begin{figure}[t]
\centering
\begin{tikzpicture}[every node/.style={font=\scriptsize}]
\begin{axis}[width=8.6cm,height=5.6cm,axis lines=middle,
  xlabel={$s$},ylabel={$t$},xlabel style={anchor=west},ylabel style={anchor=south},
  xmin=-1.6,xmax=2.0,ymin=-2.6,ymax=2.4,xtick={-1,1},ytick={-2,-1,1,2},
  clip=false]
\addplot[dark-red,thick,domain=-1.61:2.4,samples=300,variable=\u]
  ({(2-\u*\u)/(\u+2)},{\u});
\addplot[gray,dashed,domain=-1.6:2.0] {-2};
\node[gray,anchor=north east] at (axis cs:1.95,-2.08) {$t=-2$, the pole of $f$, removed};
\addplot[only marks,mark=*,mark size=1.6pt,dark-red] coordinates {(0,1.4142)(0,-1.4142)};
\addplot[only marks,mark=o,mark size=1.8pt,dark-red,thick] coordinates {(1,0)(1,-1)};
\node[anchor=north east,align=right] at (axis cs:-0.22,1.25) {$s=0$: the point $\Q(\sqrt2)$,\\ determinant $4(1+\sqrt2)$};
\node[anchor=south west] at (axis cs:1.05,0.1) {$s=1$: $t=0$, det $2$};
\node[anchor=north west] at (axis cs:1.05,-1.05) {$t=-1$, det $-1$};
\draw[gray!60,dotted] (axis cs:0,-2.6) -- (axis cs:0,2.2);
\draw[gray!60,dotted] (axis cs:1,-2.6) -- (axis cs:1,2.2);
\end{axis}
\end{tikzpicture}
\caption{The splitting pencil of \cref{ex:sqrt2}: the zeros of
$t^2+st+2s-2$ as $s$ varies (drawn over $\R$). The quadratic point
$\Q(\sqrt2)$ at $s=0$ degenerates to the rational points $0,-1$ at $s=1$;
the support stays finite over the $s$-line because the polynomial is monic
in $t$, and the pole $t=-2$ of $f$ is never met.}
\label{fig:pencil}
\end{figure}

\subsection{\texorpdfstring{$R$}{R} reduces everything to rational points and contains Witt's relations}
Write $\Z[k^\times]\subseteq\Au(\Gamma_k)$ for the span of the $(k,v)$.

\begin{proposition}\label{prop:generate}
The subgroups $\Z[k^\times]$ and $R$ together generate $\Au(\Gamma_k)$.
\end{proposition}

\begin{proof}
Induct on $d=[L:k]$ for a generator $(L,u)$. With $A$ the minimal
polynomial of a primitive element $\theta$ and $B$ the unique polynomial
of degree $<d$ with $B(\theta)=u/A'(\theta)$ in $L$, put $f=B/A$. At $p=(A)$,
$v_p(f)=-1$, $w_p=B(\theta)$ and $\varepsilon_p=A'(\theta)$, so the term is
$(L,u)$; the other points of odd valuation are the odd-order roots of $B$,
of degree $\le\deg B<d$, and possibly $\infty$, where the term is
$-(k,\mathrm{lc}(B))$ with $\mathrm{lc}(B)$ the leading coefficient of
$B$. So $(L,u)=R(f)+r(f)h_{\!A}-\Sigma$ with $\Sigma$
supported in degrees $<d$.
\end{proof}

\begin{proposition}\label{prop:MHinR}
For $u,v,s\in k^\times$ the following lie in $R$: \textup{(a)}
$(k,us^2)-(k,u)$; \textup{(b)} $(k,u)+(k,v)-(k,u+v)-(k,uv(u+v))$ for
$u+v\ne0$; \textup{(c)} $(k,u)+(k,-u)-(k,1)-(k,-1)$.
\end{proposition}

\begin{proof}
In each case we write down a rational function $f$ with rational zeros
and poles, read off $\rho(f)$ from \cref{def:R} --- at a rational point
$a$ the Euler unit is $\varepsilon_a=1$ and $w_a(f)$ is the value at $a$
of $f/(t-a)^{v_a(f)}$ --- and then simplify $R(f)$ using the relations
already established. We prove (c), then (a) using (c), then (b) using
both.

(c) Let $f=(t-a)(t-b)$ with $a\ne b$. The points of odd valuation are the
two simple zeros, with $w_a=a-b$ and $w_b=b-a$; at $\infty$ the valuation
$-2$ is even. So $\rho(f)=(k,\lambda)+(k,-\lambda)$ with $\lambda=a-b$,
of rank $2$, and $R(f)=(k,\lambda)+(k,-\lambda)-h_{\!A}$. Since $a-b$
takes every value in $k^\times$, this is (c) with $u=\lambda$.

(a) Let $f=\lambda(t-a)/(t-b)$ with $a\ne b$ and $\lambda\in k^\times$.
The odd points are the zero $a$, with $w_a=\lambda/(a-b)$, and the pole
$b$, with $w_b=\lambda(b-a)$; the valuation at $\infty$ is $0$. Writing
$w=\lambda/(a-b)$ we have $\lambda(b-a)=-\lambda^2/w$, so
$R(f)=(k,w)+(k,-\lambda^2/w)-h_{\!A}$. Subtracting the instance
$(k,w)+(k,-w)-h_{\!A}$ of (c) leaves
\[
(k,-\lambda^2/w)-(k,-w)\ =\ \big(k,(-w)(\lambda/w)^2\big)-(k,-w)\ \in\ R .
\]
Given $u,s\in k^\times$, choose $a-b=s$ and $\lambda=-us$; then $w=-u$
and $\lambda/w=s$, and the displayed element is $(k,us^2)-(k,u)$.

(b) Let $m=u+v\ne0$ and $q=v/m$; then $q\ne0,1$ because $u,v\ne0$. Put
$f=m(t-q)/\big(t(t-1)\big)$. The odd points are the zero $q$, the poles
$0$ and $1$, and $\infty$, where $f$ has a simple pole; the values are
\[
w_0=mq=v,\qquad w_1=m(1-q)=u,\qquad w_q=\frac{m}{q(q-1)}=-\frac{m^3}{uv},
\qquad w_\infty=m .
\]
Hence $\rho(f)=(k,u)+(k,v)+(k,-m^3/uv)-(k,m)$, of rank $2$, and
$R(f)=\rho(f)-h_{\!A}$. Now $-m^3/uv=(-muv)\cdot(m/uv)^2$, so by (a)
$(k,-m^3/uv)\equiv(k,-muv)$ modulo $R$, and by (c) applied to $muv$,
$h_{\!A}\equiv(k,muv)+(k,-muv)$ modulo $R$. Substituting both into
$R(f)\equiv0$ and canceling $(k,-muv)$ gives
$(k,u)+(k,v)-(k,m)-(k,muv)\in R$, which is (b).
\end{proof}

\begin{proof}[Proof of \cref{thm:reciprocity}]
The inclusion $R\subseteq\ker d$ is \cref{prop:RinKer}. If $x\in\ker d$, write
$x=y+z$ with $y\in\Z[k^\times]$, $z\in R$ (\cref{prop:generate});
then $y\in\ker d\cap\Z[k^\times]$, which is generated by Witt's relations,
and these lie in $R$ by \cref{prop:MHinR}. So $\ker d=R$, an
ideal; the second statement is \cref{thm:main}.
\end{proof}

No counting hypothesis enters: \cref{prop:MHinR} uses only
rational points with prescribed differences, which any field supplies, and
perfectness is used only so that the $\kappa(p)$ are separable, keeping the
classes $(\kappa(p),-)$ inside $\Au(\Gamma_k)$. The content of the theorem
is that the divisorial relations $R(f)$ generate everything, including the
transfer relations of \cref{prop:generate}, which are the
arithmetic shadow of the splitting pencils of \cref{sec:marks}.

\section{Complements}\label{sec:complements}

\subsection{Why there is a kernel}\label{sec:tomdieck}
No framed point dies on its own: along an $\A^1$-family of framed rational
points the Jacobian determinant is a unit of $k[t]$, hence constant, and
every relation in $\ker d$ --- already $\form u=\form{us^2}$, in
\cref{prop:relations}(a) --- involves several points at once. A kernel of
this kind is present in Segal's setting too, invisible only because his
source records the canonical framing alone. Let $C_2$ act on $\R$ by the
sign and let $0\in\R^{\mathrm{sgn}}$ carry the opposite framing
$x\mapsto-x$: the resulting element of $\pi_0^{C_2}(\sph)$ has underlying
degree $-1$ and fixed-point degree $+1$, marks $(-1,1)=(1,1)-(2,0)$, so it
equals $[\mathrm{pt}]-[C_2/e]$. A non-canonically framed point is a
combination of canonically framed orbits of different types, and the
family $x^3-tx$ realizes the relation by orbits changing type
(\cref{fig:pitchfork}). Had Segal's source for $C_2$ been decorated by
signs, his map would have been onto with this kernel.

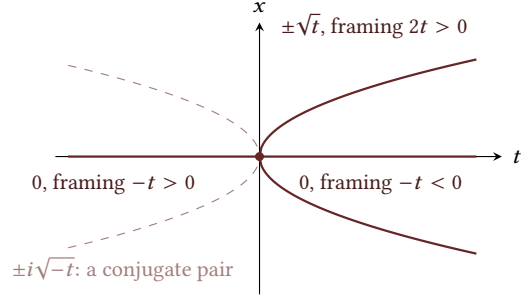
\begin{figure}[t]
\centering
\begin{tikzpicture}[scale=1.05, every node/.style={font=\scriptsize}]
\begin{axis}[width=7.2cm,height=5cm,axis lines=middle,xlabel={$t$},ylabel={$x$},
  xlabel style={anchor=west},ylabel style={anchor=south},
  xmin=-1.6,xmax=1.9,ymin=-1.85,ymax=1.8,xtick=\empty,ytick=\empty,clip=false]
\addplot[dark-red,thick,domain=-1.5:1.7,samples=2] {0};
\addplot[dark-red,thick,domain=0:1.7,samples=120] {sqrt(x)};
\addplot[dark-red,thick,domain=0:1.7,samples=120] {-sqrt(x)};
\addplot[dark-red!60,dashed,domain=-1.5:0,samples=120] {sqrt(-x)};
\addplot[dark-red!60,dashed,domain=-1.5:0,samples=120] {-sqrt(-x)};
\node[anchor=north east,dark-red!70!black] at (axis cs:-0.42,-0.08) {$0$, framing $-t>0$};
\node[anchor=north,dark-red!70!black] at (axis cs:0.95,-0.08) {$0$, framing $-t<0$};
\node[anchor=south,dark-red!70!black] at (axis cs:0.9,1.45) {$\pm\sqrt t$, framing $2t>0$};
\node[anchor=south east,dark-red!60] at (axis cs:-0.12,-1.82) {$\pm i\sqrt{-t}$: a conjugate pair};
\fill[dark-red] (axis cs:0,0) circle (1.6pt);
\end{axis}
\end{tikzpicture}
\caption{The family $x^3-tx$. For $t<0$ the zeros are the rational point
$0$ with positive framing and a conjugate pair; for $t>0$ the pair has
landed on the real line and split into two positively framed points, and
the point it landed on has reversed its framing. Read in
$\pi_0^{C_2}(\sph)$ this is $(\mathrm{pt},\text{opposite framing})=
[\mathrm{pt}]-[C_2/e]$; read in $\Au(\Gamma_\R)$, in the notation of
\cref{sec:real}, it is $(\C,z)+(\R,+)=(\R,-)+2(\R,+)$.}
\label{fig:pitchfork}
\end{figure}

\subsection{The real case}\label{sec:real}
Over $\R$, write $(\R,\pm)$ for $(\R,u)$ with $\operatorname{sgn}u=\pm1$:
modulo $\ker d$ only the sign of $u$ survives (\cref{prop:relations}(a))
and all the classes $(\C,z)$ coincide (\cref{thm:reciprocity}), while
$\GW(\R)$ is detected by rank and signature. Equivariant Betti realization \cite{HKO11,
HellerOrmsby} sends $(\R,u)$, the map $x\mapsto ux$ of $T$, to $z\mapsto uz$
on $S^{\C}=S^{1+\sigma}$, of underlying degree $1$ and fixed-point degree
$\operatorname{sgn}u$, and $(\C,z)$ to a free orbit; so the composite
$\Au(\Gamma_\R)\to\pi_0^{C_2}(\sph)\cong A(C_2)\to\Z\times\Z$ is
$(\rk,\operatorname{sgn})\circ d$: \emph{the Dress map is the mark
homomorphism read through realization}, and $(\R,-1)\mapsto[C_2/e]-
[\mathrm{pt}]$, the negative of the element of \cref{sec:tomdieck}.
Restricted to the undecorated $A(\Gamma_\R)=\Z\{[\R],[\C]\}$, $\PT$ is an
isomorphism onto $\GW(\R)$, $[\R]\mapsto\form1$, $[\C]\mapsto h$, and
realization identifies it with Segal's. The kernel of $d$ is generated by
$(\R,u)\sim(\R,uc^2)$ and $(\C,z)\sim(\R,+)+(\R,-)$, the latter being
$R(t^2+1)$; the family $x^3-tx$ realizes it as an $\A^1$-homotopy in which
a conjugate pair lands on the real line, splits into two positive points,
and flips the framing of the point it lands on:
$(\C,z)+(\R,+)=(\R,-)+2(\R,+)$. Reciprocity is the degree: for $f\in\R(t)
^\times$ nonconstant, $\rk R(f)=0$ because $\deg\operatorname{div}f=0$,
and $\operatorname{sgn}d(\rho(f))$ is the signed count of preimages of $0$
minus that of $\infty$ under $f\colon\PP^1(\R)\to\PP^1(\R)$ (a real zero
of odd order has local degree $\operatorname{sgn}w_p$, a pole
$-\operatorname{sgn}w_p$ since the chart $1/f$ reverses orientation, and at
$\infty$ the chart $1/t$ reverses it again), which vanishes because the
degree of a map of circles does not depend on the regular value
(\cref{fig:degree}). Over a
general $k$ the signatures at the orderings are the marks at the real
closures; they detect $\W(k)$ modulo torsion \cite[Ch.\ VIII,
\S3]{Lam05}, but the torsion --- already $\form{2,-1}-h\in\W(\Q)$ in
\cref{ex:sqrt2} --- is invisible to every signature, is seen by $w$
through the discriminant, and is matched on the homotopy side by the
resultant. This is the precise sense in which (S5) is harder over a field:
the invariant playing the role of the marks has torsion, and the torsion
is where the arithmetic lives.

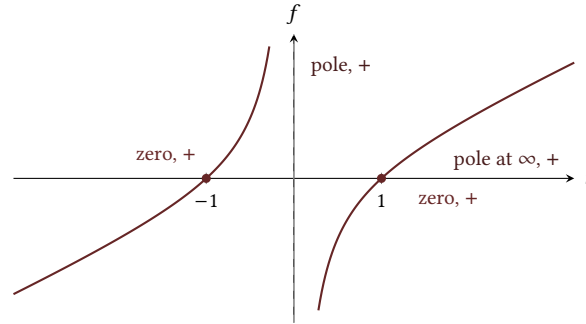
\begin{figure}[t]
\centering
\begin{tikzpicture}[every node/.style={font=\scriptsize}]
\begin{axis}[width=9cm,height=5.4cm,axis lines=middle,xlabel={$t$},ylabel={$f$},
  xlabel style={anchor=west},ylabel style={anchor=south},
  xmin=-3.2,xmax=3.2,ymin=-3.6,ymax=3.6,xtick={-1,1},ytick=\empty,clip=true]
\addplot[dark-red,thick,domain=-3.2:-0.28,samples=150] {(x^2-1)/x};
\addplot[dark-red,thick,domain=0.28:3.2,samples=150] {(x^2-1)/x};
\draw[gray!60,dashed] (axis cs:0,-3.6) -- (axis cs:0,3.6);
\fill[dark-red] (axis cs:-1,0) circle (1.7pt) node[anchor=south east,yshift=2pt] {zero, $+$};
\fill[dark-red] (axis cs:1,0) circle (1.7pt);
\node[dark-red,anchor=north west] at (axis cs:1.3,-0.12) {zero, $+$};
\node[dark-red!70!black,anchor=south west] at (axis cs:0.08,2.3) {pole, $+$};
\node[dark-red!70!black,anchor=north east] at (axis cs:3.15,0.95) {pole at $\infty$, $+$};
\end{axis}
\end{tikzpicture}
\caption{Reciprocity over $\R$ is the degree. For $f=(t^2-1)/t$ the signed
preimages of $0$ (the zeros $\pm1$, local degree $+1$ each) and of $\infty$
(the pole at $0$, local degree $-\operatorname{sgn}w_0=+1$, and the pole at
$\infty$, local degree $+\operatorname{sgn}w_\infty=+1$) both count to the
degree $2$ of $f\colon\PP^1(\R)\to\PP^1(\R)$; their difference, the
signature of $d(\rho(f))$, is zero.}
\label{fig:degree}
\end{figure}

\end{document}